\documentclass{article}
\usepackage{amssymb,amsmath}
\usepackage{graphics}

\def\qed{\hfill {\hbox{${\vcenter{\vbox{               
   \hrule height 0.4pt\hbox{\vrule width 0.4pt height 6pt
   \kern5pt\vrule width 0.4pt}\hrule height 0.4pt}}}$}}}

\def\tr{\triangleright}

\newtheorem{theorem}{Theorem}
\newtheorem{definition}{Definition}
\newtheorem{lemma}[theorem]{Lemma}
\newtheorem{proposition}[theorem]{Proposition}
\newtheorem{corollary}[theorem]{Corollary}
\newtheorem{example}{Example}
\newtheorem{remark}[example]{Remark}

\newenvironment{proof}[1][Proof]{\smallskip\noindent{\bf #1.}\quad}%
{\qed\par\medskip}

\date{}

\title{\Large \textbf{On rack polynomials}}

\author{Tim Carrell \and Sam Nelson} 

\begin{document}
\maketitle

\begin{abstract}
We study rack polynomials and the link invariants they define. We
show that constant action racks are classified by their generalized
rack polynomials and show that $ns^at^a$-quandles are not classified 
by their generalized quandle polynomials. We use subrack polynomials 
to define enhanced rack counting invariants, generalizing the quandle 
polynomial invariants.
\end{abstract}



\section{\large \textbf{Introduction}}

In \cite{N}, a two-variable polynomial invariant of finite quandles
was introduced. This polynomial quantifies the way in which
the trivial action of one quandle element on another is distributed
throughout the quandle as opposed to concentrated in a single identity
element as in a group. 

In \cite{N2} the quandle polynomial was generalized to a family of
$N^2$ polynomials where $N$ is the least common multiple of the exponents
of the columns of the quandle matrix considered as elements of of the
symmetric group $S_n$ on the elements of the quandle.
In both cases, the quandle polynomials were used to enhance the quandle
counting invariants to obtain new invariants which specialize to the
original quandle counting invariants but contain more information.

In this paper we study the natural generalization of quandle polynomials
to finite racks. We are able to show that for at least one class of
finite racks, the generalized rack polynomials determine the rack
structure up to isomorphism, and we identify another class of quandles 
for which the generalized rack polynomials do not determine the isomorphism
class. We then use these rack polynomials to enhance the rack counting
invariants from \cite{N3}.

The paper is organized as follows. In section \ref{rrp} we review the
definitions of racks and rack polynomials and give some examples. In
section \ref{grpcar} we show that constant action racks are classified 
by their generalized rack polynomials. In section \ref{nsata} we show 
that Alexander quandles have quandle polynomial of the form $ns^at^a$ and 
that unlike constant action racks, Alexander quandles and other quandles 
with quandle polynomial $ns^at^a$ are not classified by their generalized 
quandle polynomials. In section \ref{rpinv} we define rack polynomial 
enhanced counting invariants. In section \ref{quest} we collect some 
questions for future investigation.

\section{\large \textbf{Racks and rack polynomials}} \label{rrp}

In \cite{J}, Joyce defined a kind of self-distributive algebraic
structure which he dubbed a ``quandle.'' In \cite{FR}, quandles
were generalized to a larger class of self-distributive algebraic
systems known as ``racks.'' Both concepts appear under other names
in the literature such as ``distributive groupoids,'' 
``automorphic sets'' and ``kei.'' See \cite{M,B,T}.

\begin{definition}
\textup{A \textit{rack} is a set $X$ with a binary operation 
$\tr:X\times X\to X$ satisfying
\begin{list}{}{}
\item[(i)]{for all $x,y\in X$ there is a unique $z\in X$ satisfying
$x=z\tr y$, and}
\item[(ii)]{for all $x,y,z\in X$ we have $(x\tr y)\tr z=(x\tr z)\tr(y\tr z)$.}
\end{list} 
A rack which additionally satisfies
\begin{list}{}{}
\item[(0)]{for all $x\in X$, we have $x\tr x = x$}
\end{list}
is a \textit{quandle}.}
\end{definition}

Axiom (i) requires that each element $x$ of a rack $X$ acts on $X$ bijectively,
while axiom (ii) requires these bijections to be automorphisms
of the rack structure. The bijectivity of the action of $x$ gives us a right
inverse action $\tr^{-1}:X\times X\to X$ defined by $x\tr^{-1} y= z$ where
$x=z\tr y$. The reader can check that $(X,\tr^{-1})$ is also a rack, called
the \textit{dual} of $(X,\tr).$

Standard examples of rack and quandle structures include:
\begin{list}{$\bullet$}{}
\item{any union of conjugacy classes in a group $G$ with operation
$a\tr b = b^{-n}ab^n$, $n\in \mathbb{Z}$}
\item{the set of right cosets in a group $G$ of a subgroup 
$H\subset G$ (not necessarily normal) fixed by an automorphism
$s:G\to G$ with rack operation $Hx\tr Hy= s(HxHy^{-1})Hy$ }
\item{any set $X$ with a permutation $\sigma\in S_{X}$ with 
$x\tr y = \sigma(x)$ (these are \textit{constant action racks} or 
\textit{permutation racks})}
\item{the subset of a vector space $V$ on which a bilinear form
$\langle \mathbf{x},\mathbf{x}\rangle\ne 0$ with 
\[\mathbf{x}\tr\mathbf{y}=\alpha\left(\mathbf{x} 
-2\frac{\langle\mathbf{x},\mathbf{y}\rangle}{\langle\mathbf{y},
\mathbf{y}\rangle}\mathbf{y}\right)\]
where $\alpha$ is a non-zero scalar
(these are called \textit{Coxeter racks}; see \cite{FR, NW})}
\item{any module over $\mathbb{Z}[t^{\pm 1},s]/s(1-t-s)$ with $x\tr y= tx+sy$.}
\end{list}

Racks of the last type in which $s=1-t$ are known as \textit{Alexander 
quandles}. A rack is \textit{abelian} if for all $x,y,z,w\in X$ we have
\[(x\tr y)\tr(z\tr w) =(x\tr z)\tr(y\tr w).\]
In addition to being right-distributive, abelian quandles are also 
left-distributive, since we have
\[a\tr(b\tr c)= (a\tr a)\tr (b\tr c) = (a\tr b)\tr (a\tr c).\] 

A quandle is a \textit{crossed set} (see \cite{AG}) if we have
\[x\tr y = x \iff y\tr x = y.\]
The reader can check that Alexander quandles are abelian 
and Coxeter quandles (set $\alpha=-1$) are crossed sets.

We can express rack structures on a finite set $X=\{x_1,\dots, x_n\}$
in an algebra-agnostic way, i.e. without needing a formula for
$x\tr y$, by giving the rack operation table as a matrix $M_X$ whose
$(i,j)$ entry is $k$ where $x_k=x_i\tr x_j$. We call this the
\textit{rack matrix} of $X$.

\begin{example}
\textup{The constant action rack on $X=\{1,2,3\}$ with $\sigma=(132)$
has rack matrix}
\[M_X=\left[\begin{array}{ccc} 
3 & 3 & 3 \\
1 & 1 & 1 \\
2 & 2 & 2 \\
\end{array}\right].\]
\end{example}

For a rack $X$, say that an equivalence relation $\sim$ on $X$ is a
\textit{congruence} if $x\sim x'$ and $y\sim y'$ imply $x\tr y\sim x'\tr y'$.
The set $X/\sim$ of equivalence classes then forms a \textit{quotient rack}
under the operation $[x\tr y]=[x]\tr[y]$. See \cite{R} for more.

Next, we have a definition from \cite{N2}:
\begin{definition}
\textup{Let $X$ be a finite rack. For each $x\in X$, define}
\[C_m(x) = \{ y\in X \ |\ y\tr^{m} x = y\}\quad 
\mathrm{and} \quad R_n(x) = \{ y \in X\ |\ x\tr^{n} y = x\}\]
\textup{where}
\[x \tr^i y = (\dots (x\tr y) \tr y) \dots \tr y\]
\textup{where $i$ is the number of triangles. Denote $c_m(x)=|C_m(x)|$
and $r_n(x)=|R_n(x)|$. Then the $(m,n)$--\textit{rack polynomial} of $X$ 
(or $(m,n)$--\textit{quandle polynomial} if $X$ is a quandle) is}
\[rp_{m,n}(X)=\sum_{x\in X}s^{c_m(x)}t^{r_n(x)}.\]
\textup{The terms ``rack polynomial'' and ``quandle polynomial'' without
specified $m$ and $n$ values will refer to the case $m=n=1$.}
\end{definition}

\begin{example}
\textup{The constant action rack $X$ with rack matrix 
$M_X=\left[\begin{array}{ccc} 
2 & 2 & 2 \\
1 & 1 & 1 \\
3 & 3 & 3 \\
\end{array}\right]$
has rack polynomial $rp_{1,1}(X)=2t+s^3t$.}
\end{example}

\begin{remark}
\textup{In \cite{N} example 8, it is incorrectly stated that a rack
may have rack polynomial equal to zero, since in that example we have a 
contribution of $s^0t^0$ from each element. Of course, $s^0t^0=1\ne 0$, and
indeed the coefficients of a rack polynomial always sum to the cardinality 
$|X|$. The second listed author is grateful to the first for catching this
oversight.}
\end{remark}

\section{\large \textbf{Generalized rack polynomials of constant action racks}}
\label{grpcar}

In this section we show that constant action racks are classified by their
generalized rack polynomials.

\begin{proposition}\label{GenRPofConstAct}
Let $X$ be the constant action rack of a given permutation $\sigma$ on 
$\{ x_1, x_2, \dotsc , x_k\}$. Then the generalized rack polynomial of $X$ is
\[ rp_{m,n}(X) = bs^kt^a + (k-b)t^a \]
where $a$ is the number of $x_i$ such that $\sigma^n(x_i) = x_i$ and $b$ is 
the number of $x_i$ such that $\sigma^m(x_i) = x_i$.
\end{proposition}

\begin{proof}
For any $x\in \{ x_1, x_2, \dotsc , x_k\}$, 
\[
c_n(x) = \lvert\{ y \ \vert\  y\tr^n x = y \}\rvert 
= \lvert\{ y \ \vert\  \sigma^n(y)=y \}\rvert = a
\]
so each term of $rp_{m,n}(X)$ contains a factor of $t^a$. 
Furthermore,
\[ r_m(x) = \lvert\{ y\ \vert\  x\triangleright^m y = x\}\rvert 
= \lvert\{ y\ \vert\ \sigma^m(x) = x\}\rvert  =
\begin{cases}
0 & \text{if $\sigma^m(x)\neq x$,} \\
k & \text{if $\sigma^m(x)= x$.}
\end{cases}
\]
Therefore, each $x$ such that $\sigma^m(x) = x$ corresponds to a term of 
$s^kt^a$, and each $x$ such that $\sigma^m(x) \neq x$ corresponds to a term 
of $t^a$. There are $b$ distinct $x$ such that $\sigma^m(x)=x$, and thus 
$k-b$ distinct $x$ such that $\sigma^m(x)\neq x$. Therefore,
\[
rp_{m,n}(R) = bs^kt^a + (k-b)t^a.
\]
\end{proof}

Note that the original ``non-generalized''
rack polynomial corresponds to $rp_{1,1}(X)$. For this specific case, we 
have $a$ and $b$ both equal to the number of fixed points of $\sigma$, and 
so we have

\begin{corollary}\label{RPofConstAct}
Let $X$ be the constant action rack of a given permutation $\sigma$ on 
$\{ x_1, x_2, \dotsc , x_k\}$, and suppose $\sigma$ has $b$ fixed points. 
Then the rack polynomial of $X$ is
\[
rp(X) = bs^kt^b + (k-b)t^b.
\]
\end{corollary}

\begin{proposition} \label{RPCompleteInvariant}
The set of generalized rack polynomials is a complete invariant of constant 
action racks.
\end{proposition}

To prove this result we will need a pair of lemmas.

\begin{lemma} \label{Isom=Cycles}
Suppose $X$ and $X'$ are constant action racks given by permutations 
$\sigma$ and $\sigma'$, respectively. Then $X$ and $X'$ are isomorphic if 
and only if $\sigma$ and $\sigma'$ have the same cycle structure.
\end{lemma}

\begin{proof}
Suppose that $X$ and $X'$ are isomorphic. Then for any positive integer $l$,
\[
rp_{l,1}(X) = rp_{l,1}(X')
\]
and thus $X$ and $X'$ must have the same number of cycles whose length 
divides $l$. For $l=1$, this means that $X$ and $X'$ have the same number of 
cycles of length $1$. Proceeding inductively, $X$ and $X'$ have the same 
number of cycles of all positive integer lengths $l$, and so have the same 
cycle structure.

Conversely, suppose that $\sigma$ and $\sigma'$ have the same cycle structure, 
that is, we can write them as
\[
\sigma = (1_1 1_2 \ldots 1_{n_1})(2_1 2_2 \ldots 2_{n_2}) 
\ldots (k_1 k_2 \ldots k_{n_k}) \]
and
\[
\sigma' = (1_1' 1_2' \ldots 1_{n_1}')(2_1' 2_2' \ldots 2_{n_2}') 
\ldots (k_1' k_2' \ldots k_{n_k}').
\]
Let $\alpha: \sigma\to\sigma'$ be the bijection given by $a_i\mapsto a_i'$. 
Then
\[\alpha(a_i \tr b_j) = \alpha(\sigma(a_i)) = \alpha(a_{i+1}) 
= a_{i+1}' = \sigma'(a_i') = a_i' \tr b_j' 
= \alpha(a_i) \tr \alpha(b_j)
\]
and so $\alpha$ is a rack isomorphism from $X$ to $X'$.
\end{proof}

\begin{lemma} \label{RP=Cycles}
Suppose $X$ and $X'$ are constant action racks given by the permutations 
$\sigma$ and $\sigma'$, respectively. Then $rp_{m,n}(X) = rp_{m,n}(X')$ 
for all $m,n\in\mathbb{Z}^+$ if and only if $\sigma$ and $\sigma'$ have the 
same cycle structure.
\end{lemma}

\begin{proof}
By proposition \ref{GenRPofConstAct}, if $X$ and $X'$ are constant action
racks of the 
same cardinality, 
then $rp_{m,n}(X) = rp_{m,n}(X')$ if and only if
\begin{align}
  \lvert \{ x_i\ |\ \sigma^m(x_i) = x_i\} \rvert 
= \lvert \{ x_i\ |\ \sigma'^m(x_i) = x_i\} \rvert \label{sms'm}
\end{align}
and
\begin{align}
  \lvert \{ x_i\ |\ \sigma^n(x_i) = x_i\} \rvert 
= \lvert \{ x_i\ |\ \sigma'^n(x_i) = x_i\} \rvert. \label{sns'n}
\end{align}

If $\sigma$ and $\sigma'$ have the same cycle structure, then they have the 
same number of elements in cycles whose length divides $m$ and the same 
number of elements in cycles whose length divides $n$. Thus, both 
\eqref{sms'm} and \eqref{sns'n} hold and so $rp_{m,n}(X) = rp_{m,n}(X')$.

If $\sigma$ and $\sigma'$ do not have the same cycle structure, then there 
exists some minimal $l$ such that $\sigma$ and $\sigma'$ do not have the 
same number of cycles of length $l$. Thus, for $m=l$ \eqref{sms'm} cannot 
hold, and so $rp_{m,n}(X) \neq rp_{m,n}(X')$.
\end{proof}

Combining lemmas \ref{Isom=Cycles} and \ref{RP=Cycles} immediately gives 
us proposition \ref{RPCompleteInvariant}. The following example shows that 
the generalized rack 
polynomials must be used to get a complete invariant on constant action 
racks; the $(1,1)$--rack polynomial is not sufficient. 

\begin{example} 
\textup{The racks with rack matrices}
\begin{align*}
M_X=
\begin{bmatrix}
2 & 2 & 2 & 2 & 2 & 2 \\
1 & 1 & 1 & 1 & 1 & 1 \\
4 & 4 & 4 & 4 & 4 & 4 \\
3 & 3 & 3 & 3 & 3 & 3 \\
6 & 6 & 6 & 6 & 6 & 6 \\
5 & 5 & 5 & 5 & 5 & 5 
\end{bmatrix}
\qquad\mathrm{and}\qquad
M_Y=\begin{bmatrix}
2 & 2 & 2 & 2 & 2 & 2 \\
3 & 3 & 3 & 3 & 3 & 3 \\
1 & 1 & 1 & 1 & 1 & 1 \\
5 & 5 & 5 & 5 & 5 & 5 \\
6 & 6 & 6 & 6 & 6 & 6 \\
4 & 4 & 4 & 4 & 4 & 4 
\end{bmatrix}
\end{align*}
\textup{both have no fixed points and so by corollary \ref{RPofConstAct} 
have the 
same rack polynomial $rp(X)=rp(Y)=6$, but have different cycle structures 
and so by lemma \ref{Isom=Cycles} are not isomorphic.}
\end{example}

\section{\large \textbf{Quandle Polynomials of Alexander Quandles}} 
\label{nsata}

In this section we study the quandle polynomials of Alexander quandles.

\begin{definition}\textup{
Let $Q$ be an Alexander quandle. Say that $x$ and $y$ are 
\textit{$(1-t)$-equivalent}, denoted $x\sim_{(1-t)} y$ or just $x\sim y$, if
$(1-t)x=(1-t)y$.}
\end{definition}

\begin{proposition}\label{Alex=>nsata}
Let $Q$ be an Alexander quandle such that $\lvert Q\rvert = n$. Then
\[
rp(Q) = ns^at^a
\]
for some positive integer $a\vert n$.
\end{proposition}

\begin{proof}
Let $\varphi : Q\to Q$ be the function given by
\begin{align*}
\varphi(q) &= 0\tr q.
\end{align*}
Since
\[ 0\triangleright q = t\cdot 0 + (1-t)\cdot q = (1-t)\cdot q,\]
$\varphi$ is just left multiplication by $1-t$. Since Alexander quandles 
are abelian and thus left distributive, $\varphi$ is automatically a 
homomorphism. By the first isomorphism theorem $\varphi$ partitions $Q$ 
into cosets by the congruence 
\[q\sim p \iff \varphi(q)=\varphi(p).\] These cosets are all of size 
$a=\lvert \ker\varphi \rvert$; clearly $a\vert n$. We will show that
 $q\tr p = q$ if and only if $q\sim p$.

Suppose $q\tr p = q$. Then
\begin{align*}
tq + (1-t)p &= q \\
(t-1)q + (1-t)p &= 0 \\
(1-t)(p-q) &= 0
\end{align*}
so $p-q\in\ker\varphi$ and so $q\sim p$.

Suppose $q\sim p$. Then $0 = \varphi(p-q)$ implies
\[
0 = \varphi_0(p-q) = (1-t)(p-q) =(t-1)q + (1-t)p 
\]
and hence
\[q = tq + (1-t)p = q\triangleright p.\]

Therefore, for any $q\in Q$ there are exactly $a$ choices of $p$ such 
that $q\tr p = q$ and also exactly $a$ choices of $p$ such that 
$p\tr q = p$. Thus, $r(q)=c(q)=a$ for all $q\in Q$, and so
\[
rp(Q) = ns^at^a.
\]
\end{proof}

Note that Proposition \ref{Alex=>nsata} provides a second, equivalent 
definition for $(1-t)$-equivalence.
\begin{corollary}\label{cor7}
An equivalent definition of $\sim$ is $p\sim q$ if and only if $p\tr q=p$.
\end{corollary}

This also gives us another proof of the fact noted in \cite{AG} that
\begin{corollary}
All Alexander quandles are crossed sets. 
\end{corollary}
\begin{proof}
This follows immediately from corollary \ref{cor7} and the 
reflexivity of equivalence relations.
\end{proof}

These results suggest the possibility of defining $\sim$ for all crossed 
sets, or proving the converse of \ref{Alex=>nsata}. However, there is a 
counterexample to both of these natural conjectures.
\begin{example} \label{CE}
\textup{Let $Q$ be the crossed set with matrix}
\begin{align*}
M_Q = \begin{bmatrix}
1 & 3 & 2 & 1 & 1 & 1 \\
3 & 2 & 1 & 2 & 2 & 2 \\
2 & 1 & 3 & 3 & 3 & 3 \\
4 & 4 & 4 & 4 & 6 & 5 \\
5 & 5 & 5 & 6 & 5 & 4 \\
6 & 6 & 6 & 5 & 4 & 6
\end{bmatrix}.
\end{align*}
\textup{
Suppose, as for Alexander quandles and $ns^at^a$ quandles, we define $\sim$ 
by $p\sim q$ when $p\tr q = p$. Then for $Q$ we have $1\sim 4$ and $4\sim 2$, 
but $1 \not\sim 2$. Therefore, $\sim$ cannot be an equivalence relation for 
this crossed set.}

\textup{Furthermore, 
\begin{align*}
(1 \tr 1) \tr (4 \tr 2) = 1 \tr 4 = 1,
\end{align*}
but
\begin{align*}
(1 \tr 4) \tr (1 \tr 2) = 1 \tr 3 = 2
\end{align*}
and so $Q$ is not abelian and thus not Alexander, disproving the converse 
of proposition \ref{Alex=>nsata}.}
\end{example}
This also provides a counterexample to the conjecture in \cite{N2} that 
distinct non-Latin quandles are distinguished by at least one of their 
generalized rack polynomials. 
\begin{example}
\textup{Let $R$ be the quandle with quandle matrix}
\begin{align*}
M_R = \begin{bmatrix}
1 & 1 & 2 & 2 & 1 & 1 \\
2 & 2 & 1 & 1 & 2 & 2 \\
3 & 3 & 3 & 3 & 4 & 4 \\
4 & 4 & 4 & 4 & 3 & 3 \\
6 & 6 & 5 & 5 & 5 & 5 \\
5 & 5 & 6 & 6 & 6 & 6
\end{bmatrix}.
\end{align*}
\textup{Both $R$ and $Q$ from example \ref{CE} have the generalized rack 
polynomial}
\[
rp_{(m,n)}(R) =rp_{(m,n)}(Q) = 6s^ct^d
\]
\textup{with}
\begin{align*}
c = \begin{cases}
6 & \mathrm{when}\ $n$\ \mathrm{is\ even,} \\
4 & \mathrm{when}\ $n$\ \mathrm{is\ odd,}
\end{cases}
\qquad\mathrm{and}\qquad
d = \begin{cases}
6 & \mathrm{when}\ $m$\ \mathrm{is\ even,} \\
4 & \mathrm{when}\ $m$\ \mathrm{is\ odd.}
\end{cases}
\end{align*}
\textup{However, $R$ is abelian and $Q$ is not abelian, so $Q$ and $R$ are not 
isomorphic.}
\end{example}

\section{\large \textbf{Rack polynomial enhanced link invariants}}\label{rpinv}

In \cite{N3}, the quandle counting invariant $|\mathrm{Hom}(Q(L),T)|$ was
extended to the case of finite non-quandle racks. In this section we will
enhance this invariant with rack polynomials. We begin by recalling how this
was done in the quandle case.

\begin{definition}
\textup{Let $S$ be a subrack $S\in X$. The $(m,n)$--\textit{subrack polynomial}
is} \[srp^{m,n}_{S\subset X}(s,t)=\sum_{x\in S} s^{c_m(x)}t^{r_n(x)}.\] 
\end{definition}

The subquandle polynomials of the image subquandles in $\mathrm{Hom}(Q(L),T)$ 
are used to enhance the quandle counting invariants in \cite{N,N2}. 
Specifically, instead of counting $1$ for each element of 
$\mathrm{Hom}(Q(L),T)$ to obtain the quandle counting invariant
$|\mathrm{Hom}(Q(L),T)|$, we count $sr_{\mathrm{Im}(f)\subset T}(s,t)$
to obtain a multiset of subquandle polynomials. We can express these multisets
in a polynomial-style form by writing the elements of the multiset as
powers of a variable $z$ and the multiplicities as coefficients.

\begin{definition}
Let $L$ be a link and $T$ a finite quandle. The $(m,n)$--\textit{subquandle 
polynomial invariant} of $L$ with respect to $T$ is then
\[sp_{m,n}(L,T)=\sum_{f\in \mathrm{Hom}(Q(L),T)} z^{srp^{m,n}_{\mathrm{Im}(f)\subset T}(s,t)}.
\]
\end{definition}

Now, let $L$ be an oriented link with $c$ ordered components. For any diagram 
$D$ of $L$, we can regard $D$ as a framed link using the blackboard framing, 
i.e. giving each component of $L$ a framing number $w_i$ equal to its 
self-writhe. Thus, such a diagram has a framing vector 
$\mathbf{w}=(w_1,\dots,w_c)\in \mathbb{Z}^c.$

For any finite rack $T$, let $N(T)$ be the \textit{rack rank} of $T$,
i.e. the exponent of the permutation in $S_{|T|}$ along the diagonal of
the rack matrix of $T$. If two ambient isotopic diagrams of $D$
have writhe vectors which are componentwise congruent modulo $N(T)$,
then there is a bijection 
\[\phi: \mathrm{Hom}(FR(D,\mathbf{w}),T)\to\mathrm{Hom}(FR(D,\mathbf{w'}),T)\]
between the sets of rack homomorphisms from the fundamental racks of 
$(D,\mathbf{w})$ and $(D,\mathbf{w'})$ into $T$ defined by sending a coloring
of one diagram to a coloring of the same diagram with $mN$ kinks added. 
Indeed, since any subrack containing an element
$x\in T$ must also contain the rack powers $x^{\tr n}$ for all 
$n\in \mathbb{Z}$ (see \cite{N3}), $\phi$ preserves image subracks.
Hence, as far as $T$ is concerned, the framing
vectors of $D$ live in $W=(\mathbb{Z}_{N(T)})^c$, and we have an invariant
of unframed links given by
\[SR(L,T)=\left|\{ f\in \mathrm{Hom}(FR(D,\mathbf{w}),T) \ |
\ \mathbf{w}\in W\}\right|,\]
called the \textit{simple rack counting invariant}. A refinement obtained
by keeping track of which colorings belong to which framings is the 
\textit{rack counting polynomial} invariant
\[PR(L,T)=\sum_{\mathbf{w}\in W} \left(
|\mathrm{Hom}(FR(D,\mathbf{w}),T)|\prod_{i=1}^cq_i^{w_i}\right).\]
When $T$ is a quandle, $N(T)=1$ and we have 
$SR(L,T)=|PR(L,T)|=|\mathrm{Hom}(Q(L),T)|$. See
\cite{N3} for more.

We would like to jazz up these rack counting invariants with the generalized
rack polynomials. To this end, we propose the following

\begin{definition}
\textup{Let $L$ be a link of $c$ components, $T$ a finite rack with rack
rank $N(T)$, and $W=(\mathbb{Z}_{N(T)})^c$. Then the $(m,n)$--\textit{simple 
subrack polynomial enhanced rack counting multiset} is the multiset}
\[srpm_{m,n}(L,T)=\left\{ \left. srp^{m,n}_{\mathrm{Im}(f)\subset T}(s,t)
\ \right| \ \mathbf{w}\in W,\ f\in \mathrm{Hom}(FR(D,\mathbf{w}),T)
\right\}\]
\textup{and the $(m,n)$--\textit{subrack 
polynomial enhanced rack counting multiset} is the multiset of ordered pairs}
\[rpm_{m,n}(L,T)=\left\{ \left.
\left(srp^{m,n}_{\mathrm{Im}(f)\subset T}(s,t),
\prod_{i=1}^c q_i^{w_i}\right) \ \right| 
\ \mathbf{w}\in W,\
f\in \mathrm{Hom}(FR(D,\mathbf{w}),T)
\right\}\]
\textup{We can also define the invariants in a more polynomial-style form
for ease of comparison:}
\[srpp_{m,n}(L,T)=\sum_{\mathbf{w}\in W}
\left(\sum_{\ f\in \mathrm{Hom}(FR(D,\mathbf{w}),T)}
z^{srp^{m,n}_{\mathrm{Im}(f)\subset T}(s,t)}\right)\]
\textup{and}
\[rpp_{m,n}(L,T)=
\sum_{\mathbf{w}\in W}\left(\sum_{\ f\in \mathrm{Hom}(FR(D,\mathbf{w}),T)}
\prod_{i=1}^{c} q_i^{w_i}z^{srp^{m,n}_{\mathrm{Im}(f)\subset T}(s,t)}\right).\]
\end{definition}

Specializing $s=t=0$ (or, indeed, $z=1$) in the subrack polynomial $rpp$ 
yields the rack counting
polynomial. Since every finite quandle $T$ is a rack, the fact that 
subquandle polynomial invariants are stronger than unenhanced quandle counting 
invariants means \textit{a fortiori} that subrack polynomial invariants
are stronger than unenhanced rack counting invariants. The next example shows
how subrack enhancement gives more information about a knot or link than
the unadorned rack counting invariant.

\begin{example}
\textup{The trefoil knot $3_1$ has simple rack counting invariant value 
$20$ with respect to the rack with rack matrix below.}
\[\begin{array}{cc}
\includegraphics{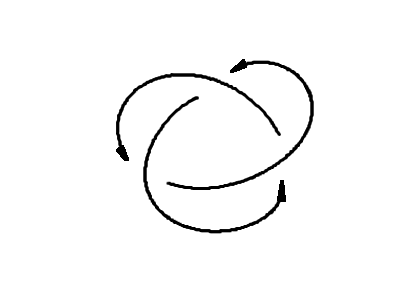} & \includegraphics{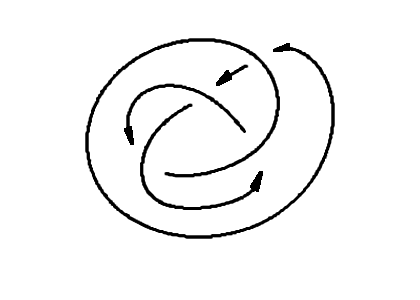} \\
\mathrm{odd \ writhe} & \mathrm{even \ writhe}
\end{array}\quad \quad
M_T=\left[\begin{array}{ccccc} 
1 & 3 & 2 & 1 & 1 \\
3 & 2 & 1 & 2 & 2 \\
2 & 1 & 3 & 3 & 3 \\
4 & 4 & 4 & 5 & 5 \\
5 & 5 & 5 & 4 & 4
\end{array}\right]
\]
\textup{Here $N(T)=2$, so we need only consider two diagrams of $3_1$, one
with even writhe and one with odd writhe.
The rack counting polynomial here is $11+9q$, which says that
11 colorings are contributed from the even-writhe diagram and 9 are 
contributed by the odd-writhe diagram.
The subrack polynomial invariant is $rpp(3_1,T)=2z^{2s^3t^3} + 3z^{s^3t^3} 
+ 6z^{3s^3t^3} + 3qz^{s^3t^3} + 6qz^{3s^3t^3}$, which further filters the 
contributions -- of the nine colorings of the odd-writhe diagram, six use 
colors in the subrack with subrack polynomial $3s^3t^3$ (in this case, the 
subquandle $\{1,2,3\}$) while three have colors in subracks with subrack 
polynomial $s^3t^3$ (here, the singleton subquandles $\{1\},\{2\},\{3\}$).
Similarly, the even-writhe diagram has colorings corresponding to the
odd-writhe colorings as expected, but additionally has two colorings by
the subrack $\{4,5\}.$}
\end{example}

Indeed, the example suggests the following

\begin{proposition}
If $S\subset T$ is a quandle and $K$ a knot, then the contributions to 
$rpp(K,T)$ from $S$ are equal for all powers of $q$. That is, $rpp$ 
includes the term
\[\left(\sum_{f\in \mathrm{Hom}(Q(K),S)}z^{sp_{\mathrm{Im}(f)\subset T}}\right)
(1+q+\dots+q^{N(T)-1}).\]
\end{proposition}

\begin{proof}
Quandle colorings do not depend on framing, so we get the same contribution,
namely
\[\sum_{f\in \mathrm{Hom}(Q(K),S)}z^{sp_{\mathrm{Im}(f)\subset T}},\]
from each framing.
\end{proof}

\section{\large \textbf{Questions}}\label{quest}

We have shown that some classes of racks are classified by their generalized
rack polynomials (the constant action racks) while others are not
(quandles with polynomial $ns^at^a$). What conditions are sufficient for a 
type of rack to be determined by its generalized rack polynomials?

Since $ns^at^a$ quandles are not determined by their generalized quandle 
polynomials, what extra information is necessary to determine these quandles 
up to isomorphism? How can such extra information be incorporated into the 
enhanced rack counting invariants?

For every rack $R$, the quotient rack under operator equivalence 
($x\sim y \iff z\tr x = z\tr y \ \forall z$) is a quandle. What is the 
relationship between the subrack polynomial invariant with respect to $R$ 
and the subquandle polynomial invariant with respect to $Q=R/\sim$?

\bigskip

\textsc{Department of Mathematics, Pomona College,
 610 N. College Ave, Claremont, CA 91711}

\noindent
\textit{Email address: }\texttt{tnc02005@mymail.pomona.edu}

\bigskip

\textsc{Department of Mathematics, Claremont McKenna College,
 850 Colubmia Ave., Claremont, CA 91711}

\noindent
\textit{Email address: }\texttt{knots@esotericka.org}

\end{document}